\documentclass{imsart}
\RequirePackage[OT1]{fontenc}
\usepackage{amssymb, amsmath, amsthm}
\RequirePackage[numbers]{natbib}
\RequirePackage[colorlinks,citecolor=blue,urlcolor=blue]{hyperref}

\usepackage{graphicx}
\usepackage{tikz}
\usetikzlibrary{arrows,automata}
\usepackage{enumerate}
\usepackage{dsfont}

\startlocaldefs
\numberwithin{equation}{section}
\theoremstyle{plain}
\newtheorem{thm}{Theorem}[section]
\endlocaldefs
\newtheorem{proposition}[thm]{Proposition}

\newtheorem{lemma}[thm]{Lemma}
\newtheorem{corollary}[thm]{Corollary}
\newtheorem{remark}{Remark}[section]

\usepackage[T1]{fontenc}
\usepackage{lmodern}

\newcommand{\N}{\ensuremath{{\mathbb N}}}

\newcommand{\R}{\ensuremath{{\mathbb R}}}

\newcommand{\e}{\ensuremath{{\rm e}}}
\DeclareMathOperator*{\inte}{int}

\newcommand{\E}{\ensuremath{{\mathbb E}}}
\newcommand{\Pro}{\ensuremath{{\mathbb P}}}

\newcommand{\ov}[1]{\overline{#1}}

\begin{document}

\begin{frontmatter}
\title{Riesz representation and optimal stopping with two case studies}
\runtitle{Riesz representation and optimal stopping}
\begin{aug}
\author{
\snm{S\"oren Christensen,}\ead[label=e1]{sorenc@chalmers.se}}
\author{ \snm{Paavo Salminen}\ead[label=e2]{phsalmin@abo.fi}\thanks{Research supported in part by a grant
      from Svenska kulturfonden via Stiftelsernas professorspool, Finland}}
\runauthor{S. Christensen and P. Salminen}

\affiliation{Chalmers University of Technology and G\"oteborg University}
\address{S\"oren Christensen\\Department of Mathematical Sciences\\ Chalmers University of Technology and\\ G\"oteborg University\\
SE-412 96 G\"oteborg\\ Sweden\\
\printead{e1}}

\address{Paavo Salminen\\
Faculty of Science and Engineering\\
\AA bo Akademi University\\
 FIN-20500, \AA bo\\
 Finland\\
\printead{e2}\\
}
\end{aug}

\begin{abstract}
In this paper we demonstrate that {the Riesz representation of
excessive functions is a useful and enlightening tool to study optimal
stopping problems.}  
After a short general discussion of the Riesz
representation we concretize, firstly, on a $d$-dimensional and, secondly, a
 space-time one-dimensional geometric Brownian motion. After this, two classical
optimal stopping problems are discussed: 1) the optimal investment problem and 2) the
valuation of the American put option. It is seen in both of these
problems that the boundary of the stopping region can be characterized
 as a unique solution of
an integral equation arising immediately from the Riesz representation of the value
function.  { In Problem 2 the derived equation coincides
  with the standard well-known equation found in the literature.}
\end{abstract}



\end{frontmatter}

\section{Introduction}
\label{intro}
An optimal stopping problem (OSP) can be formulated as follows:
Find a function $V$ (value function) and a stopping time $\tau^*$
(optimal stopping time) such that 
\begin{equation}\label{OSP}\tag{*}
V({\bf x}):=\sup_{\tau\in\mathcal{M}}\mathbb{E}_{\bf x}\big({\rm e}^{-r\tau}g(X_\tau)\big)
=\mathbb{E}_{\bf x}\big({\rm e}^{-r\tau^*}g(X_{\tau^*})\,;\, \tau^*\leq T\big),
\end{equation}
where $(X_t)_{t\geq 0}$ is a strong Markov process taking values in
$E\subseteq \R^d,\; {\bf x}\in E,\;r\geq 0,$ $T\in(0,\infty]$ is the time horizon of
 the problem, $\mathcal{M}$ is the set of all stopping times
in the natural filtration of $X$  with values in $[0,T]$, and the function $g$
(reward function) is { often} assumed to be non-negative and continuous.
{ In case, $\tau=\infty$ in (\ref{OSP}) we define 
$$
{\rm e}^{-r\tau}g(X_\tau):=\limsup_{t\to\infty}{\rm e}^{-r\,t}g(X_t).
$$
}
 Notice that we use boldface letters to denote 
non-random vectors and matrices. 

Optimal stopping problems arise naturally in many different areas,
such as stochastic calculus (maximal inequalities), mathematical
statistics (sequential analysis), and mathematical finance (pricing of
American-type derivatives and real options), { for these applications and
  further references, see,
  e.g.. the monographs \cite{shiryayev78} and \cite{PS}}. An explicit solution for optimal stopping
problems is often hard to find. Most examples are such that the 
underlying process is one-dimensional, often a diffusion
process, and the time horizon is infinite, see
e.g. \cite{salminen85,BL00,D-K} and the references therein. In
contrast, the class of explicit examples with a multidimensional
underlying process or with finite-time horizon are very
limited.  { In this article, we describe a solution method for such problems
  based on the Riesz representation of the excessive functions.} Notice that finite-time horizon problems with one-dimensional
underlying process $\left(X_t\right)_{t\geq
 0}$ may be seen as
two-dimensional where, in fact, we use the space-time process $\left((t,X_t)\right)_{t\geq
 0}$ as the underlying. 

More precisely, we consider the classical problem of optimal timing for
an irreversible investment decision under the assumption that the
revenue and cost factors follow (possibly correlated) geometric
Brownian motions. It is furthermore  assumed that the cost factors
consist of many different sources making the model more
realistic. For the mathematical formulation, see the expression for
the value function in (\ref{optinvest1}) in
Section 3.   

This problem has been studied extensively over the last decades,
see, e.g., \cite{DS,OS,ho,NR,GS,GS2,CI11}, and the references
therein. However,  no explicit description of the optimal stopping set is
known so far to the best of our knowledge. We remark that in \cite{ho}
a closed form solution was presented under certain conditions on the
parameters and the optimal stopping time was claimed to be a hitting
time of a halfspace. Unfortunately, it turned out that this closed
form solution is only valid in trivial degenerated cases if the
dimension is greater than one, see \cite{CI11} and \cite{NR}. Because
the structure of the reward function is additive and not
multiplicative, there is no hope for such an easy solution in
dimensions $d\geq 2$. 

Our contribution hereby is
to give an implicit description of the stopping region via an integral equation
which has the boundary curve of the stopping region as a unique solution.     
It is seen in Section 3.6 that the equation has a fairly simple form
especially in the two-dimensional case. {We also present an ad hoc numerical
metod for solving the integral equation.}

The optimal investment problem in one dimension and with finite
horizon is equivalent with the optimal stopping problem for finding
the price of an  American put option. To characterize the exercise boundary
analytically and to develop numerical algorithms for finding it explicitly is an
important and much studied topic in mathematical finance with the origin in McKean \cite{mckean65}. We refer to Peskir and Shiryayev
\cite{PS} pp. 392-395 for a discussion with many references. Our main object of interest is an integral equation
for the exercise boundary derived at the beginning of the 1990s
in the papers by Kim \cite{Kim90}, Jacka \cite{jacka91},  and Carr,
Jarrow and Myneni \cite{CJM92}. { See also Myneni \cite{Myneni}, Karatzas and Shreve \cite{KaShreve98}, Peskir and Shiryayev
\cite{PS}, and Pham \cite{Pham1997}, Lamberton and Mikou \cite{Lamberton2008} for the problem with underlying jump diffusions.}
The uniqueness of the solution of the equation was proved by Peskir \cite{Peskir2005_2}
using a delicate stochastic analysis involving local times on
curves. { The method presented in this paper results to the same equation and} we offer here a proof for
the uniqueness based on the uniqueness of the representing measure in
the Riesz representation of the value function.


 To briefly motivate our approach, recall that a
non-negative, measurable function $u$ is called $r$-excessive for
$X$ if the following two conditions hold:
\[
\mathbb{E}_{\bf x}\big({\rm e}^{-r t}u(X_t)\big)\leq u({\bf x})\qquad
 \forall t\geq 0,\, {\bf x}\in E,
\]
\[
\lim_{t\rightarrow 0} \mathbb{E}_{\bf x}\big({\rm e}^{-r t}u(X_t)\big)=u({\bf x})\qquad
 \forall  {\bf x}\in E.
\]

 { Assuming that $X$ has continuous sample paths and the reward function $g$ is lower
  semicontinuous and positive satisfying the condition
$$
\E_{\bf x}\left(\sup_{t\geq 0}g(X_t) \right)<\infty
$$
it can be proved that  the value function $V$ exists and is
characterized as the smallest $r$-excessive majorant of $g,$ see
Theorem 1 p. 124 in Shiryayev \cite{shiryayev78}. Moreover,} if $g$ is continuous, the optimal stopping time is then known to be the first entrance time into the set 
\begin{equation}\label{eq:stop_set}
S:=\{{\bf x}\in E:g({\bf x})=V({\bf x})\}
\end{equation}
called the stopping region. For the finite time horizon problem,
analogous results hold for the space-time process
$\left((t,X_t)\right)_{t\geq0},$  {since herein the
  first co-ordinate can also be seen as a (deterministic) Markov process.} To utilize these basic theoretical
facts to solve explicit problems of interest, we need a good
description of $r$-excessive functions. Such a description -- the
Riesz representation -- is discussed in the following section with
emphasis on geometric Brownian motion. From Section 3 onward the
paper is organised as follows. In Section 3 we study the optimal
investment problem.  Section 4 is on American put option and the paper is concluded with an appendix where proofs of some more technical results are given. 

\section{The Riesz representation of excessive functions}\label{subsec:riesz}

Our basic tool in analyzing and solving OSP 
is
the Riesz representation of excessive functions according to which an
excessive function can be written as the sum of a potential and a
harmonic function. For thorough discussions of the Riesz representation
and related matters in a general framework of Hunt processes, see Blumenthal and Getoor \cite{BG} and Chung and
Walsh \cite{CW}. A more
detailed representation of excessive functions is derived in the
Martin boundary theory which,
in particular, provides representations also for the harmonic
functions, see Kunita and Watanabe \cite{KW} and Chung and
Walsh \cite{CW} Chapter 14. For applications of the Riesz and the Martin
representations in optimal stopping, see Salminen \cite{salminen85},
Mordecki and Salminen \cite{MS}, Christensen and Irle \cite{CI11}, and 
Crocce and Mordecki \cite{CrocceMordecki}. We also remark that in Christensen et
al. \cite{CST} an alternative representation of excessive functions
 via expected
suprema is utilized  to characterize solutions of OSPs
and, moreover, the connection with the Riesz representation
is studied.


\subsection{Multi-dimensional geometric Brownian motion}\label{subsec:riesz_infinite}
{ Let  
$$W=((W^{(1)}_t,\dots,W^{(d)}_t))_{t\geq 0}$$ 
be a $d$-dimensional Brownian motion started from $(0,\dots,0)$ such
that for $t\geq 0$
\[\E\big(W^{(i)}_t\big)=0,\;\;\E\left(\big(W^{(i)}_t\big)^2\right)=t,\;\;\E\big(W^{(i)}_tW^{(j)}_t\big)=\sigma_{ij}t,\,\,i,j=1,\dots,d.\]
It is assumed that the non-negative definite  matrix
$\Sigma:=(\sigma_{ij})_{i,j=1}^d$ with $\sigma_{ii}:=1$ is non-singular.  
A $d$-dimensional geometric Brownian motion is a diffusion $X$ in  ${\R_+^d:=(0,+\infty)^d}$ 
with the components defined by
\[X^{(i)}_t=X^{(i)}_0\exp\big(a_iW^{(i)}_t+(\mu_i-\frac{1}{2} a_i^2)t\big),\;i=1,...,d,\]
where $a_i\not= 0$ for $i=1,...,d.$ The  differential operator ${\cal G}$ associated with $X$ is of the form
\[{\cal G}:=\frac{1}{2}\sum_{i,j=1}^d\rho_{ij}x_ix_j\frac{\partial^2}{\partial x_i\partial x_j}+\sum_{i=1}^d\mu_ix_i\frac{\partial}{\partial x_i},\]
where $\rho_{ij}:= a_i a_j\sigma_{ij}.$ The (row) vector ${\boldsymbol \mu}:=(\mu_i)_{i=1}^d$ and the matrix ${\boldsymbol
  R}:=(\rho_{ij})_{i,j=1}^d$ are called the parameters of
  $X.$ }
 
To be able to apply the Riesz representation on  $X$ this
process should satisfy some regularity conditions.  Firstly,  we note 
that $X$ { is a standard Markov process, see \cite{BG}
  p. 45. Secondly, $X$ }has a resolvent kernel given by 
\begin{equation}\label{resolvent}
 G_r({\bf x},{\bf y}):=\int_0^\infty {\rm e}^{-r t}\, p(t;{\bf x},{\bf y})\, dt, 
\end{equation}
where ${\bf x}, {\bf y} \in\R_+^d$ and $(t,{\bf x},{\bf y})\mapsto p(t;{\bf x},{\bf y})$ is a transition
density of $X.$ { The following proposition shows that the transition
density may be taken with respect to a measure $m$ such that the 
corresponding resolvent is self-dual, i.e., in duality  with itself relative to $m;$   this means
that relationship (\ref{badua}) below holds (see, e.g.,
\cite{CW} p. 344).}  Note that, since $m$ turns out to be absolutely continuous with respect to the Lebesgue measure, it is a matter of standardization to choose the Green kernel with respect to $m$ or with respect to the Lebesgue measure. In the following, we consider $G_r$ with respect to $m$ and use the notation
\[G_rf({\bf x}):=\int_{\R_+^d} G_r({\bf x},{\bf y})f({\bf y})m({\bf dy}),\]
where $f$ satisfies some  appropriate measurability and integrability
conditions.

\begin{proposition}\label{prop:self-dual}
A $d$-dimensional geometric Brownian motion $X$ { as introduced above}
 is self-dual, i.e.
for all nonnegative measurable functions $f$ and $g$ it holds that
\begin{equation}
\label{badua}
\int_{\R_+^d} f({\bf x})G_rg({\bf x})m({\bf dx})
=\int_{\R_+^d} g({\bf y})G_rf({\bf y})m({\bf dy}),
\end{equation}
where $m$ is the measure on ${\R_+^d}$ with the Lebesgue density
\[h({\bf y}):=\exp\left(-\sum_{i=1}^d\log(y_i)+2\,{\boldsymbol
  {\ov{\mu}}}\,{\boldsymbol \Sigma}^{-1}
\log ({\bf y})^{tr}\right),\]
$\log({\bf y}):=(\log y_1,...,\log y_d)$, {$\boldsymbol{ \ov{\mu}}=\big((\mu_i-\frac{1}{2}a_i^2)/a_i\big)_{i=1}^d$}, and $()^{tr}$ denotes transposition.
\end{proposition}

\smallskip
\noindent
{\sl Proof.} See Section A1 in the appendix.
\smallskip

From the self duality it follows that Hypothesis
B in \cite[p. 498]{KW} holds. { Notice also that since the dual 
resolvent kernel is identical with the resolvent kernel of $X$ the
process associated with the dual resolvent is a standard Markov
process identical in law with $X.$  Consequently,  the following (strongest)
form of the Riesz
representation theorem holds.}

\begin{thm}\label{thm:riesz}
Let $u$ be a locally integrable $r$-excessive 
function for a $d$-dimensional geometric Brownian motion $X.$ 
 { Then $u$ can be represented uniquely as the sum of a (non-negative)
$r$-harmonic function $h$  and an $r$-potential $p$. For the potential $p$
there exists a unique Radon measure $\sigma$ depending on $u$ and $r$ 
on $\R_+^d$ such that for all  ${\bf x}\in\R_+^d$ 
\begin{equation}
\label{riesz} 
p({\bf x})=\int_{\R_+^d} G_r({\bf x},{\bf y})\, \sigma({\bf dy}) 
\end{equation}
Moreover, if $A$ is an open set having a compact closure in $\R_+^d$
then $u$ is $r$-harmonic on $A$ if and only if
$\sigma(A)=0.$ 
}\end{thm}

We remark that the uniqueness of $\sigma$ follows from the fact
that $X$ is a self dual standard process (see, e.g., \cite[Proposition
7.11 p. 503]{KW}). The statement about $r$-harmonicity on $A$ can be
deduced from ibid. Proposition 11.2 p. 513. Recall also that a
non-negative measurable function $u$ is called 
 $r$-harmonic on $A$ if for all ${\bf x}$
\begin{equation}
\label{harm} 
u({\bf x})=\E_{\bf x}\left( {\rm e}^{-r\tau_A}\,u(X_{\tau_A})\right),
\end{equation}
 where
$$ 
\tau_A:=\inf\{t\,:\, X_t\not\in A\}.
$$

In general, it is often difficult to find explicit expressions for the
harmonic function $h$ in the Riesz decomposition presented in Theorem
\ref{thm:riesz}. The following proposition gives an easy condition under
which $h$ vanishes.
\begin{proposition}\label{prop:h=0}
Let $u$ be a bounded $r$-excessive function for a $d$-dimensional
geometric Brownian motion $X$. Then $h\equiv 0$ in the integral representation of $u$ given in Theorem \ref{thm:riesz}.
\end{proposition}

\begin{proof}
Take a sequence $(K_n)_{n\in \N}$ of compact subsets of $\R_+^d$ such that $\tau_{n}:=\inf\{t\geq 0: X_t\not\in K_n\}\rightarrow \infty$ as $n\rightarrow\infty$. Then, due to the boundedness of $u$ and the $r$-harmonicity of $h$, it holds that
\begin{align*}
0\leq h({\bf x})=\E_{\bf x}\left({\rm e}^{-r\tau_n}h(X_{\tau_n})\right)\leq \E_{\bf x}\left({\rm e}^{-r\tau_n}u(X_{\tau_n})\right)\rightarrow 0.
\end{align*}
\end{proof}

In case $u$ is smooth enough the representing measure\ $\sigma_u$ can
be obtained by applying the differential operator $\cal G$ on $u.$ This is made precise in the next

\begin{proposition}\label{prop:repr_measure_infinite}
Let $u$ be a bounded $r$-excessive function for a $d$-dimensional geometric Brownian motion $X$ such that $u\in C^1$ and let $D\subseteq \R^d_+$ be a convex set with $u\in C^2$ on $\R_+^d\setminus \partial D$. Furthermore, assume that 
{ $({\cal G}-r)u$ is }
locally bounded around $\partial D$.
Then the representing measure $\sigma=\sigma_u$ for $u$ on $\R_+^d$ in
the integral representation \eqref{riesz} is absolutely continuous
with respect to the measure $m$ and is given for ${\bf y}\not\in \partial D$  by
\[\sigma({\bf dy})=(r-{\cal G})u({\bf y})m({\bf dy}).\]
\end{proposition}
\noindent
{\sl Proof.} See Section A2 in the appendix.

\subsection{Space-time geometric Brownian motion}
 We now consider a one-dimensional geometric Brownian motion $X$ in space-time, that is, the
 two-dimensional process $\bar X=\left((t,X_t)\right)_{0\leq t< T}$ with the state space
 $I=[0,T)\times \R_+.$ The differential operator associated with
  $\bar X$ is 
\begin{equation}
\label{gen2}
{\bar {\cal G}} :=
{\partial \over{\partial t}}+{a^2\over 2}x^2
{\partial^2 \over{\partial x^2}}+
\mu x{\partial \over{\partial x}}.
\end{equation}
 We remark that the definition of an 
  $r$-excessive function $u$ for $\bar X$ 
can be written in the
  form
$$
\lim_{t\downarrow s}\bar\E_{(s,x)}({\rm e}^{-r\,(t-s)}u(t,X_t))
\uparrow u(s,x)\qquad \forall x>0, s\in[0,T),
$$
where $\bar\E$ denotes the expectation operator associated with $\bar X.$
{
The resolvent kernel of  $\bar X$} can be defined 
as follows
\begin{equation}\label{eq:green_time_space}
\bar G_r\left((s,x),(t,y)\right)=
\begin{cases}
{\rm e}^{-r\, (t-s)}p(t-s;x,y),& 
s<t\leq T,\\
0,& t\leq s<T,\ x\not= y,\\
+\infty,& t=s<T,\ x= y,
\end{cases}
\end{equation}
where the transition density $p$ is taken  with respect to the speed measure
$m(dx)=\frac 2{a^2} x^{\gamma-2}dx,$ $ \gamma=2\mu/a^2.$ 
{
Notice that the
kernel is also defined for $t=T.$ 
\begin{proposition}\label{prop:space-time-riesz}
Let $u$ be an $r$-excessive function of $\bar X$ locally integrable
on $[0,T)\times \R_+$ with respect to $\lambda\times m,$ where
$\lambda$ denotes the Lebesgue measure. Then there exists a unique 
 Radon measure $\sigma$ on $(0,T]\times\R_+$ such that for
 $(s,x)\in[0,T)\times \R_+$ 
\begin{equation}
\label{space-time-exc-riesz}
u(s,x)= \int\int_{(0,T]\times\R_+}\bar
G_r\left((s,x),(t,y)\right)\sigma(dt,dy).
\end{equation}
\end{proposition}
\begin{proof}  It is proved 
in the appendix, see Section A3,  that there exists a (dual) resolvent
kernel $\widehat G$ such that for non-negative and measurable $f$
and $g$  
\begin{align}
\label{space-time-dual}
\int_0^T dt \int_{\R_+}m(dx) f(t,x)\bar G_rg(t,x)
=
\int_0^T dt \int_{\R_+}m(dx) g(t,x){\widehat G}_rf(t,x).
\end{align}
It can be checked then that Hypothesis (B) in Kunita and Watanabe
\cite[p. 498]{KW} holds. Consequently, see ibid Theorem 2 p. 505, $u$
has the Riesz representation 
\begin{equation}
\label{space-time-exc-riesz2}
u(s,x)= \int\int\bar
G_r((s,x),(t,y))\,\sigma^{(1)}(dt,dy)+h(s,x),
\end{equation}
where $\sigma^{(1)}$ is  a Radon measure,  $h$ is a harmonic function and the integration is over the set
$I_P$ consisting of  the points $(t,y)$ in the state space $I$ 
for which  $(s,x)\mapsto
G_r((s,x),(t,y))$ is a potential. The representation of $u$ in (\ref{space-time-exc-riesz2})
as the sum of a potential and a harmonic function is unique. Moreover,
the representing measure  $\sigma^{(1)}$ is unique, cf. \cite[p. 503]{KW} . To deduce
(\ref{space-time-exc-riesz}) notice firstly that $I_P=(0,T)\times\R_+$
since $G_r(\cdot,(t,y))$ is a
potential for all 
$(t,y)\in(0,T)\times\R_+.$ This follows readily from the definition of
$G_r,$  where it is stated that 
$
G_r((s,x),(t,y))=0
$
for $s>t.$ Secondly, applying the Martin boundary theory (we omit the
details) it can be
proved that the harmonic function $h$ in
(\ref{space-time-exc-riesz2}) has the representation
\begin{equation}
\label{space-time-har-riesz2}
h(s,x)= \int_{\R_+}\bar
G_r((s,x),(T,y))\,\sigma^{(2)}(dy),
\end{equation}
where $\sigma^{(2)}$ is a Radon measure on $\R_+.$ Also here the
representing measure $\sigma^{(2)}$ is uniquely determined by $h.$
Combining  (\ref{space-time-exc-riesz2}) and
(\ref{space-time-har-riesz2}) yields (\ref{space-time-exc-riesz}).
\end{proof}
}
\begin{remark} 
The proof of the duality does not use any particular
  properties of geometric Brownian motion (see Appendix). Consequently, the
  uniqueness of the representing measure in
  (\ref{space-time-exc-riesz}) holds for general space-time
  one-dimensional diffusions.  
\end{remark}

For $r$-excessive functions $u$ that are smooth enough, we can describe the form of the measure $\sigma$ more explicitly. The following result is useful generalization of \cite[Proposition 2.2]{Salminen99} based on Alsmeyer and Jaeger \cite{AlsmeyerJaeger}.

\begin{proposition}\label{prop:space-time-measure}
Let $u$ be a bounded $r$-excessive function for $\bar X$ on $[0,T)\times \R_+$ such that $\partial u/\partial t$ and $\partial u/\partial x$ are continuous on $(0,T)\times \R_+$, and $\partial u/\partial x$ is absolutely continuous as a function of the second argument.
Then the representing measure $\sigma$ for $u$ on $(0,T)\times \R_+$ in the representation \eqref{space-time-exc-riesz} is absolutely continuous with respect to the Lebesgue measure on $(0,T)\times \R_+$ and is given by
\[\sigma(ds,dy)=(r-\bar {\cal G})u(s,y)ds\;m(dy).\]
\end{proposition}

\begin{proof}
By \cite[Teorema 8.2]{dynkin69} (see \cite[Theorem 8.2]{dynkin69_engl} for an English translation) it holds that for all continuous functions $f$ with compact support in $(0,T)\times \R_+$ we have
\begin{align*}
&\int\int_{[0,T]\times\R_+} f(s,y)\sigma(ds,dy)\\=&\lim_{t\rightarrow0}\int\int_{[0,T]\times\R_+} f(s,y)\frac{u(s,y)-\bar\E_{(s,y)}\left({\rm e}^{-rt}u(s+t,X_{s+t})\right)}{t}ds\;m(dy).
\end{align*}
Therefore, we have to prove that for $(s,y)\in(0,T)\times \R_+$
\begin{align}\label{eq:generator}
\lim_{t\rightarrow 0}\frac{u(s,y)-\bar\E_{(s,y)}\left({\rm e}^{-rt}u(s+t,X_{s+t})\right)}{t}\end{align}
exists and is equal to $(r-\bar {\cal G})u(s,y)$. From \cite[Corollary 2.2]{AlsmeyerJaeger} it is seen that It\^o's formula can be applied to obtain $\Pro_{(s,y)}$-a.s.
\begin{align*}
u(s,y)-{\rm e}^{-rt}u(s+t,X_{s+t})=&-a\int_s^{s+t}{\rm e}^{-rv}X_v\frac{\partial}{\partial x}u(u,X_v)dW_v\\&+\int_s^{s+t}{\rm e}^{-rv}(r-\bar {\cal G})u(v,X_v)dv.
\end{align*}
Using a stopping argument and taking expectations yield the existence of the limit in \eqref{eq:generator} and, hence, the claim is proved.
\end{proof}
{
\begin{remark} The regularity assumptions in Proposition
  \ref{prop:repr_measure_infinite} are fairly strong and sometimes
  difficult to check. However, these are possible to relax by applying
  other extensions of the Ito formula (without the local time terms). 
\end{remark}
}
\subsection{Basic idea of using the Riesz representation for solving optimal stopping problems}
There is a wide range of different approaches for solving OSPs.
{ Many} of these are based on considering candidates for the value function of the form
\begin{align}\label{eq:cand_orig}
V_{\ov{S}}({\bf x})=\mathbb{E}_{\bf x}\big({\rm e}^{-r\tau_{\ov{S}}}g(X_{\tau_{\ov{S}}})\big)
\end{align}
for candidate sets $\ov{S}$ and associated first hitting times
$\tau_{\ov{S}}$, and then finding properties of the true value
function that characterize one candidate set as the optimal stopping
set. This idea can then be translated into a free-boundary problem, as
described extensively in the monograph \cite{PS}. One of the major
technical problems in using this approach is that a priori the
candidate functions $V_{\ov{S}}$ are typically not smooth on the
boundary of $\ov{S}$, so that it is not straightforward to apply tools
such as It\^o's formula 
or Dynkin's lemma. 

Our idea for treating OSPs  
using the Riesz representation theorem can basically be described as follows (for the infinite time horizon): Using the general results presented earlier in this section, we first show that the value function $V$ can be written in the form
\[V({\bf x})=\int_{S} G_r({\bf x},{\bf y})\, \hat\sigma({\bf y})m({\bf dy})\]
for some known function $\hat\sigma$. Now, in contrast to \eqref{eq:cand_orig}, we characterize the unknown stopping set $S$ by considering candidates for the value function of the form
\begin{equation}\label{eq:cand_riesz}
V_{\ov{S}}({\bf x})=\int_{\ov{S}} G_r({\bf x},{\bf y})\, \hat\sigma({\bf y})m({\bf dy})
\end{equation}
and then identify one candidate set as the optimal stopping set. From a technical point of view, these candidate solutions are easy to handle, since the strong Markov property immediately yields that a variant of Dynkin's formula holds true for all $\ov S$, see Lemma \ref{lem:integr_stop} below.

To show the applicability of this approach for treating concrete
problems of interest, we concentrate  in this article on two case
studies, namely the multidimensional optimal investment problem with
infinite time horizon in Section \ref{sec:invest} and the American put
problem in Section \ref{sec:put}. { Notice that the latter problem can
be viewed as the optimal investment problem under a finite
time horizon  with $d=1$ .}

\section{Optimal investment problem}\label{sec:invest}

{
In this section we concentrate on one of the most famous OSPs 
in continuous time with multidimensional underlying process: the optimal investment
problem, which goes back to \cite{DS}.
}
The value function associated with the optimal investment problem is given by
\begin{equation}
\label{optinvest1}
v_{\boldsymbol \alpha}({\bf x}):=\sup_{\tau\in\mathcal{M}}\E_{{\bf
    x}}({\rm e}^{-r\tau}(X^{(0)}_\tau-\alpha_1X^{(1)}_\tau-...-\alpha_dX^{(d)}_\tau)^+
    ),
~~{\bf x}\in\R_+^{d+1}.
\end{equation}
Here, we assume the time horizon to be infinite, i.e. $T=\infty$,
$r>0$, $\boldsymbol\alpha \in\R_+^d$ is a weight vector, and
 { $(X^{(0)}_t,X^{(1)}_t,...,X^{(d)}_t)_{t\geq 0}$ is a $d+1$-dimensional}
geometric Brownian motion { as defined in Subsection
  \ref{subsec:riesz_infinite}}  { with the indices $i$ and $j$ running
from 0 to $d.$} 
As discussed in \cite{NR} and \cite{CI11}, we furthermore assume that $r>\mu_0$ to guarantee the value function to be finite and the optimal stopping time not to be infinite a.s. 

\subsection{Problem reduction}
First, by the explicit dependence of $X$ on the starting point, we see that $v_{\boldsymbol \alpha}({\bf x})=v_{\bf{1}}(x_0,\alpha_1x_1,...,\alpha_dx_d)$ for all ${\boldsymbol \alpha}\in\R_+^d,\;{\bf x}\in\R_+^{d+1}$, where we write ${\bf 1}=(1,...,1)$. 
Therefore, we may take ${\boldsymbol \alpha}=(1,....,1)$. 
Because the reward function is homogeneous, it is 
{
standard
}
to reduce the dimension of the problem, see
e.g. \cite{NR}. { To recall
this briefly, notice that} for all $\tau$ and all ${\bf x}$ it holds that
\begin{align*}
&\E_{{\bf_x}}\left({\rm e}^{-r\tau}(X^{(0)}_\tau -  X^{(1)}_\tau-...-X^{(d)}_\tau)^+
\right)\\
&\hskip1cm =x_0\tilde{\E}_{\bf x}\left({\rm e}^{-(r-\mu_0)\tau}(1-\tilde{X}^{(1)}_\tau-...-\tilde{X}^{(d)}_\tau)^+
\right),
\end{align*}
where $\tilde{X}^{(i)}={X}^{(i)}/{X}^{(0)}$ are geometric Brownian motions under the measure $\tilde{\Pro}$ given by
\[\frac{d\tilde{\Pro}}{d\Pro}\big|_{\mathcal{F}_t}=\frac{{\rm e}^{-\mu_0t}}{x_0}X^{(0)}_t.\]
To be more explicit, under  $\tilde{\Pro},$ $\tilde{X}^{(i)}$ has drift $\mu_i-\mu_0$ and
volatility 
$a_0^2+a_i^2-2a_ia_0\sigma_{0i}$. Consequently, we may, without loss of
generality, take $X^{(0)}\equiv K$ (a positive constant). 

To summarize, we consider the optimal stopping problem
\begin{equation}\label{eq:value}
V({\bf x})=\sup_{\tau} \E_{\bf x}\left({\rm e}^{-r\tau}\left(K-\sum_{i=1}^dX_\tau^{(i)}\right)^+
\right),\;\;\;{\bf x}\in\R_+^d,
\end{equation}
that is, an optimal stopping problem of the form \eqref{OSP} with \[g({\bf x})=\left(K-\sum_{i=1}^d x_i\right)^+,\]
{
where, for notational convenience, the problem \eqref{eq:value} is
formulated for $X$ under $\Pro$ (instead of the transformed process $\tilde{X}$).
A typical assumption in the literature for this problem
 is that $\mu_i\leq r$ for $i=1,...,d$, see \cite{ho,OS}. This
 guarantees that the optimal stopping time is a.s. finite. Since some
 arguments can be shortened (see e.g. Lemma \ref{lem:properties_S}),
 we also use this assumption throughout Section 3.
}

\subsection{Preliminary results}
We first collect some elementary results of the optimal stopping set
$S$ as defined in \eqref{eq:stop_set}.
{
Similar results and lines of argument can also be found in \cite{OS,Paulsen2001}.
}

\begin{lemma}\label{lem:properties_S}
\begin{enumerate}[(i)]
\item $S$ is a subset of
\[\left\{{\bf x}\in\R_+^d:K-\sum_{i=1}^d x_i>0\mbox{ and }(r-{\cal G})g({\bf x})\geq 0\right\}.\]
\item $S$ is a closed convex set. 
\item $S$ is south-west-connected, that is if ${\bf x}\in S$, then so is ${\bf y}$ for all 
{
${\bf y}\leq {\bf x}$, where we understand $\leq$ componentwise.
}
\end{enumerate}
\end{lemma}
\begin{proof}
For $(i)$ note that if $K-\sum_{i=1}^d x_i\leq 0$, then the reward for
immediate stopping in ${\bf x}$ is 0; since, obviously, $V>0$, ${\bf x}$
cannot be in the optimal stopping set. Furthermore, if $(r-{\cal G})g({\bf
  x})<0$, then $g$ is $r$-subharmonic in a neighborhood of ${\bf x}$,
so that at ${\bf x}$ it is also not optimal to stop. {Moreover, since
 $g$ and $V$ are continuous (for the latter claim, see,
e.g., \cite{BG} p. 85) it follows that $S=\{g=V\},$
cf. (\ref{eq:stop_set}),  is closed.} For convexity, take
${\bf x},{\bf y}\in S,$ $\lambda\in(0,1)$, and let $\tau$ be a stopping
time. 
 Then we have 
\begin{align*}
\E_{\bf 1}\left({\rm e}^{-r\tau}(K-(\lambda {\bf x}+(1-\lambda){\bf
  y})X^{tr}_\tau)^+\right)\leq \,
&\lambda \E_{\bf 1}\left({\rm e}^{-r\tau}(K-{\bf x}X^{tr}_\tau)^+\right)\\
\,&\hskip.5cm+(1-\lambda) \E_{\bf 1}\left({\rm e}^{-r\tau}(K-{\bf y}X^{tr}_\tau)^+\right)\\
\leq\,& \lambda V({\bf x})+(1-\lambda) V({\bf y})\\
=\,&\lambda \left(K-\sum_{i=1}^d x_i\right)+(1-\lambda)\left(K-\sum_{i=1}^d y_i\right)\\
=\,&K-\sum_{i=1}^d(\lambda x_i+(1-\lambda)y_i)\\
\leq\,& g(\lambda {\bf x}+(1-\lambda){\bf y}),
\end{align*}
which proves that $\lambda {\bf x}+(1-\lambda){\bf y}\in S$, i.e., $S$ is convex. 
{
For $(iii)$, note that for ${\bf x}\in S$ and ${\bf y}\leq {\bf x}$, it holds for all stopping times $\tau$
\begin{align*}
\E_{\bf 1}\left({\rm e}^{-r\tau}(K-{\bf y}X^{tr}_\tau)^+\right) \,
\leq \,& \E_{\bf 1}\left({\rm e}^{-r\tau}(K-{\bf x}X^{tr}_\tau)^+\right)+\E_{\bf 1}\left({\rm e}^{-r\tau}({\bf x}-{\bf y})X^{tr}_\tau\right)\\
\leq\,& V({\bf x})+\sum_{i=1}^d (x_i-y_i)\\
=\,& K-\sum_{i=1}^d x_i+\sum_{i=1}^d (x_i-y_i)\\
=\,& g({\bf y}),
\end{align*}
where we used for the second inequality that, by our assumption on the
drift, i.e.,
$\mu_i\leq r$ for all $i,$ and since ${\bf x}\geq {\bf y}$, the process $({\rm e}^{-rt}({\bf x}-{\bf y})X^{tr}_t)_{t\geq 0}$ is a nonnegative supermartingale. This proves that ${\bf y}\in S$. 
}
\end{proof}


\subsection{Integral representation of the value function}
By the general theory of $r$-excessive functions described in Section \ref{subsec:riesz}, we know that the value function $V$ has the representation 
\begin{equation}\label{eq:int_repr1}
V({\bf x})=\int_SG_r({\bf x},{\bf y})\sigma({\bf dy})+h({\bf x}),
\end{equation}
where $h$ is an $r$-harmonic function. This is our starting point for solving the optimal stopping problem \eqref{eq:value}. 
%
%
 We check first that the value function has
enough regularity. This is formulated in the next lemma.

\begin{lemma}\label{lem:smooth_fit}
It holds that $V\in
C^1$ and $V|_{\R_+^d\setminus \partial S}\in C^2$. Furthermore, $({\cal G}-r)V$ is locally bounded around $\partial S$,
{
i.e. for each $x_0\in \partial S$, there exists $\delta>0$ such that $({\cal G}-r)V$ is bounded on $\{x\in (0,\infty)^d\setminus \partial S: |x-x_0|<\delta\}.$
}
\end{lemma}
\noindent
{\sl Proof.} See Section A4 in the appendix.

\smallskip
\noindent
Using now the results obtained in Subsection
\ref{subsec:riesz_infinite} we obtain the explicit form of the
integral representation of the value function.

\begin{thm}\label{fund_prop}
For all ${\bf x}\in \R_+^d$ it holds that 
\begin{equation}\label{eq:int_repr}
V({\bf x})=\int_SG_r({\bf x},{\bf y})\hat\sigma({\bf y})m({\bf dy})\;\;\;\left(=G_r\hat\sigma({\bf x})\right),
\end{equation}
where \[\hat\sigma({\bf y}):=(r-{\cal G})g({\bf y})=rK+\sum_{i=1}^d(\mu_i-r)y_i.\]
\end{thm}

\begin{proof}
First note that in the representation \eqref{eq:int_repr1}, the measure $\sigma$ vanishes on $S^c$ since $V$ is $r$-harmonic on the continuation set. Since $g$ is bounded, so is $V$. Using this fact together with Lemma \ref{lem:smooth_fit} and Lemma \ref{lem:properties_S}, Propositions \ref{prop:h=0} and \ref{prop:repr_measure_infinite} are applicable and yield that $h=0$ and 
on $\inte(S)$
\[\sigma({\bf dy})=(r-{\cal G})V({\bf y})m({\bf dy}).\]
But since $V=g$ on $\inte(S)$, we obtain 
\[\sigma({\bf dy})=(r-{\cal G})g({\bf y})m({\bf dy})=\left(rK+\sum_{i=1}^d(\mu_i-r)y_i\right)m({\bf dy}),\]
which gives the result. 
\end{proof}
{
Evaluating \eqref{eq:int_repr} at $\partial S$ we obtain the following
corollary. 
\begin{corollary}\label{coro:boundary}
For all ${\bf x}\in \partial S$
\begin{equation}\label{eq:int_eqn}
g({\bf x})=\int_SG_r({\bf x},{\bf y})\hat\sigma({\bf y})m({\bf dy}).
\end{equation}
\end{corollary}
}
{
Note that a description of the stopping boundary in
{\eqref{eq:int_eqn} is very natural} and similar results arise in many other treatments of optimal stopping problems, in particular with a finite time horizon, see, e.g., the examples given \cite{PS}.
}

\subsection{Uniqueness of the solution of the integral equation}
{ In the previous section we have found the identity
  \eqref{eq:int_eqn} which can be seen as an equation for the unknown
  boundary of the stopping set.} When analyzing this equation from a
purely analytical point of view, there does not seem to be much hope
that this equation would characterize the manifold uniquely.
{ However,} using a probabilistic reasoning based on our integral representation,
we now show that $\partial S$ is indeed uniquely determined by
\eqref{eq:int_eqn}. More precisely we prove
\begin{thm}\label{thm:uniqueness}
%
Let $\ov S$ be a nonempty, south-west connected, convex set such that  
\[\ov S\subseteq \{{\bf x}\in\R_+^d:\sum_{i=1}^dx_i<K\}\]
and assume that for  all ${\bf x}\in \partial \ov S$ it holds
\begin{equation}
\label{assumption_1}
g({\bf x})=\int_{\ov S}G_r({\bf x},{\bf y})\hat\sigma({\bf y})m({\bf dy}).
\end{equation}
Then $\ov S=S$. 
\end{thm}

In the proof of this theorem, we frequently make use of the following
{ well-known }version of Dynkin's formula for functions of form
\eqref{eq:int_repr1}. The proof is an easy application of the strong
Markov property and can be found in Section A5 in the appendix.
\begin{lemma}\label{lem:integr_stop}
Let $s:\R_+^d\rightarrow[0,\infty)$ be a measurable function and
\[
w({\bf x}):=\int G_r({\bf x},{\bf y})s({\bf y})m({\bf dy}).\]
Then for each stopping time $\tau$ and each ${\bf x}\in\R_+^d$
\begin{equation}
\label{dynkin_1}
w({\bf x})=\E_{\bf x}\left({\rm e}^{-r\tau}w(X_\tau)\right)+\E_{\bf x}\left(\int_0^\tau {\rm e}^{-rt}s(X_t)dt\right).
\end{equation}
\end{lemma}

\begin{proof}[Proof of Theorem \ref{thm:uniqueness}]
Write { for ${\bf x}\in\R_+^d$
\begin{align*}
\ov V({\bf x})&=\int_{\ov S}G_r({\bf x},{\bf y})\hat\sigma({\bf y})m({\bf dy})
=\int G_r({\bf x},{\bf y})\ov\sigma({\bf y})m({\bf dy}),
\end{align*}
where $\ov\sigma({\bf y}):=\sigma({\bf y})\,1_{\ov S}(y).$}
We proceed in four steps:
\begin{enumerate}
\item\label{1} $\ov V({\bf x})=g({\bf x})$ for all ${\bf x}\in \ov S$:\\
Let $\ov \gamma:=\inf\{t\geq 0:X_t\not\in \ov S\}$. Using Lemma
\ref{lem:integr_stop} and { assumption (\ref{assumption_1}) we obtain for all ${\bf x}\in \ov S$}
\begin{align*}
\ov V({\bf x})&=\E_{\bf x}\left({\rm e}^{-r\ov \gamma}\ov V(X_{\ov
  \gamma})\right)+\E_{\bf x}\left(\int_0^{\ov \gamma}
{\rm e}^{-rt}\ov\sigma(X_t)dt\right)\\
&=\E_{\bf x}\left({\rm e}^{-r\ov \gamma}g(X_{\ov \gamma})\right)+\E_{\bf x}\left(\int_0^{\ov \gamma} {\rm e}^{-rt}\hat\sigma(X_t)dt\right)\\
&=\E_{\bf x}\left({\rm e}^{-r\ov \gamma}g(X_{\ov \gamma})\right)+\E_{\bf x}\left(\int_0^{\ov \gamma} {\rm e}^{-rt}(r-{\cal G})g(X_t)dt\right)\\
&=g({\bf x}),
\end{align*}
where we applied the continuity by monotone convergence (since $\ov \sigma$ is bounded and ${\bf x}\mapsto G_r({\bf x},{\bf y})$ is continuous for all ${\bf y}$) of $\ov V$ and $g$ in the second equality and Dynkin's formula to obtain the last equality.
\item \label{2} $\ov V({\bf x})\leq V({\bf x})$ for all ${\bf x}\in\R_+^d$:\\
For ${\bf x}\in \ov S$ the inequality holds by step \ref{1}. Now, let
${\bf x}\not\in \ov S$ and write $\ov\tau=\inf\{t\geq 0:\ X_t\in \ov
S\}$. 
{ Using again Lemma \ref{lem:integr_stop} yields}
\begin{align*}
\ov V({\bf x})&=\E_{\bf x}\left({\rm e}^{-r\ov \tau}\ov
V(X_{\ov\tau})\right)+\E_{\bf x}\left(\int_0^{\ov \tau}
{\rm e}^{-rt}\ov\sigma(X_t)dt\right)\\
&=\E_{\bf x}\left({\rm e}^{-r\ov \tau}g(X_{\ov \tau})\right)\\
&\leq V({\bf x}),
\end{align*}
since $V$ is the value function. 
\item \label{3} $S\subseteq \ov S$:\\
Let ${\bf x}\in S\cap\ov S$ and write $\gamma=\inf\{t\geq 0:\ X_t\not\in S\}$. Then by step \ref{1}, Lemma \ref{lem:integr_stop}, and step \ref{2}
\begin{align*}
g({\bf x})&=\ov V({\bf x})=\E_{\bf x}\left({\rm e}^{-r \gamma}\ov
V(X_{\gamma})\right)+\E_{\bf x}\left(\int_0^{\gamma}
{\rm e}^{-rt}\ov\sigma(X_t)dt\right)\\
&\leq \E_{\bf x}\left({\rm e}^{-r \gamma}V(X_{\gamma})\right)+\E_{\bf x}\left(\int_0^{\gamma} {\rm e}^{-rt}\hat\sigma(X_t)1_{\{X_t\in \ov S\}}dt\right).
\end{align*}
On the other hand, by Lemma \ref{lem:integr_stop} applied to $V${, i.e. to $s(x)=\hat\sigma(x)1_S(x)$,}
\begin{align*}
g({\bf x})&=V({\bf x})=\E_{\bf x}\left({\rm e}^{-r \gamma} V(X_{\gamma})\right)+\E_{\bf x}\left(\int_0^{\gamma} {\rm e}^{-rt}\hat\sigma(X_t)1_{\{X_t\in S\}}dt\right)\\
&= \E_{\bf x}\left({\rm e}^{-r \gamma}V(X_{\gamma})\right)+\E_{\bf x}\left(\int_0^{\gamma} {\rm e}^{-rt}\hat\sigma(X_t)dt\right).
\end{align*}
Subtracting yields that $\E_{\bf x}\left(\int_0^{\gamma} {\rm e}^{-rt}\hat\sigma(X_t)1_{\{X_t\not\in \ov S\}}dt\right)\leq 0$. Since this equation holds for all ${\bf x}\in S\cap\ov S$ and $\hat\sigma>0$ on $\ov S$, we obtain 
\begin{align}\label{eq:contrad_stopping}
\Pro_{\bf x}(X_t\not\in \ov S\mbox{ for some $t\leq \gamma$})=0.
\end{align}
{Now, if there would exist $y\in S\setminus \ov S$, by the closeness of $\ov S$, the boundary $B$ of the rectangular solid 
\[\{z\in(0,\infty)^d:z\leq y\}\]
has a positive surface area in $S\setminus \ov S$. By the southwest
connectedness of $S$, it holds that $\hat \gamma:=\inf\{t\geq
0:\ X_t\in B\}\leq \gamma$ and $\Pro_{\bf x}(X_{\hat \gamma}\not\in
\ov S)>0$, in contradiction to
\eqref{eq:contrad_stopping}. This 
proves $S\subseteq \ov S$.}
\item \label{4} $\ov S\subseteq S$:\\
Let $\tau^*$ be the optimal stopping time, that is, 
$\tau^*=\inf\{t\geq 0:\ X_t\in S\}.$ Then for all ${\bf x}\not\in S$ it holds by step \ref{2} and Lemma \ref{lem:integr_stop} that
\begin{align*}
V({\bf x})&\geq \ov V({\bf x})=\E_{\bf x}\left({\rm e}^{-r \tau^*}\ov
V(X_{\tau^*})\right)+\E_{\bf x}\left(\int_0^{\tau^*}
{\rm e}^{-rt}{ \ov\sigma(X_t)}dt\right).
\end{align*}
Since $X_{\tau^*}\in S\subseteq\ov S$ we obtain by step \ref{1} that $\ov V(X_{\tau^*})=g(X_{\tau^*})$, so that 
\begin{align*}
V({\bf x})&\geq \E_{\bf x}\left({\rm e}^{-r \tau^*}g(X_{\tau^*})\right)+\E_{\bf x}\left(\int_0^{\tau^*} {\rm e}^{-rt}\hat\sigma(X_t)1_{\{X_t\in \ov S\}}dt\right)\\
&=V({\bf x})+\E_{\bf x}\left(\int_0^{\tau^*} {\rm e}^{-rt}\hat\sigma(X_t)1_{\{X_t\in \ov S\}}dt\right),
\end{align*}
hence as above $\Pro_{\bf x}(X_t\not\in \ov S\mbox{ for some $t\leq \tau^*$})=0$ for all ${\bf x}\not\in S$, i.e., $\ov S\subseteq S$. 
\end{enumerate}
\end{proof}

{
\begin{remark}
In the previous proof, we use a similar structure as, e.g., in the
proof of the uniqueness in \cite{Peskir2005_2}.
\end{remark}
}

\subsection{Solution of the investment problem in case $d=1$}

To understand how Theorem \ref{thm:uniqueness} can be used to solve OSP %
\eqref{eq:value} explicitly, we
consider the case $d=1$ and $\mu=r.$ In other words,
 we consider the problem
connected to pricing a perpetual American put with strike price $K$ in a Black-Scholes
market. We refer to  \cite{mckean65} for an early treatment of the problem. Hence, the underlying process $(X_t)_{t\geq 0}$ is the geometric Brownian motion
with drift parameter $r>0$ (the risk neutral interest rate) and
volatility $a^2$
 Recall (see Borodin and Salminen \cite{B-S}) 
that the (symmetric) Green kernel with respect to the speed measure
$m(dx)=\frac 2{a^2}\, x^{\gamma-2}dx$ is given by 
\[G_r(x,y)=\frac {y\, x^{-\gamma}}{1+\gamma}\,
,\quad 0\leq y\leq x,
\]
where $\gamma:={2r/{a^2}}.$ 

Lemma \ref{lem:properties_S} shows that the optimal stopping problem
is one-sided, i.e., there exists a boundary point $x^*$ such that
$S=(0,x^*]$. Corollary \ref{coro:boundary} and Theorem
\ref{thm:uniqueness} characterize the boundary point $x^*$ as the unique solution to the equation
\[g(x^*)=\int_0^{x^*}G_r(x^*,y)\hat\sigma(y)m(dy),\]
where $g(x)=(K-x)^+$ and $\hat\sigma(y)=rK.$ 
By straightforward calculations we obtain 
\begin{equation}\label{eq:opt_value_1d}
K-x^*=\frac 1{1+\gamma}\,(x^*)^{-\gamma}\int_0^{x^*}
y\,\, rK\,\, \frac{2}{a^2}\, y^{\gamma-2}\,dy=\frac K{1+\gamma}
\end{equation}
which yields the well-known solution $x^*= \gamma K/(1+\gamma).$

\begin{figure}[t]
\center
\includegraphics[width=8cm]{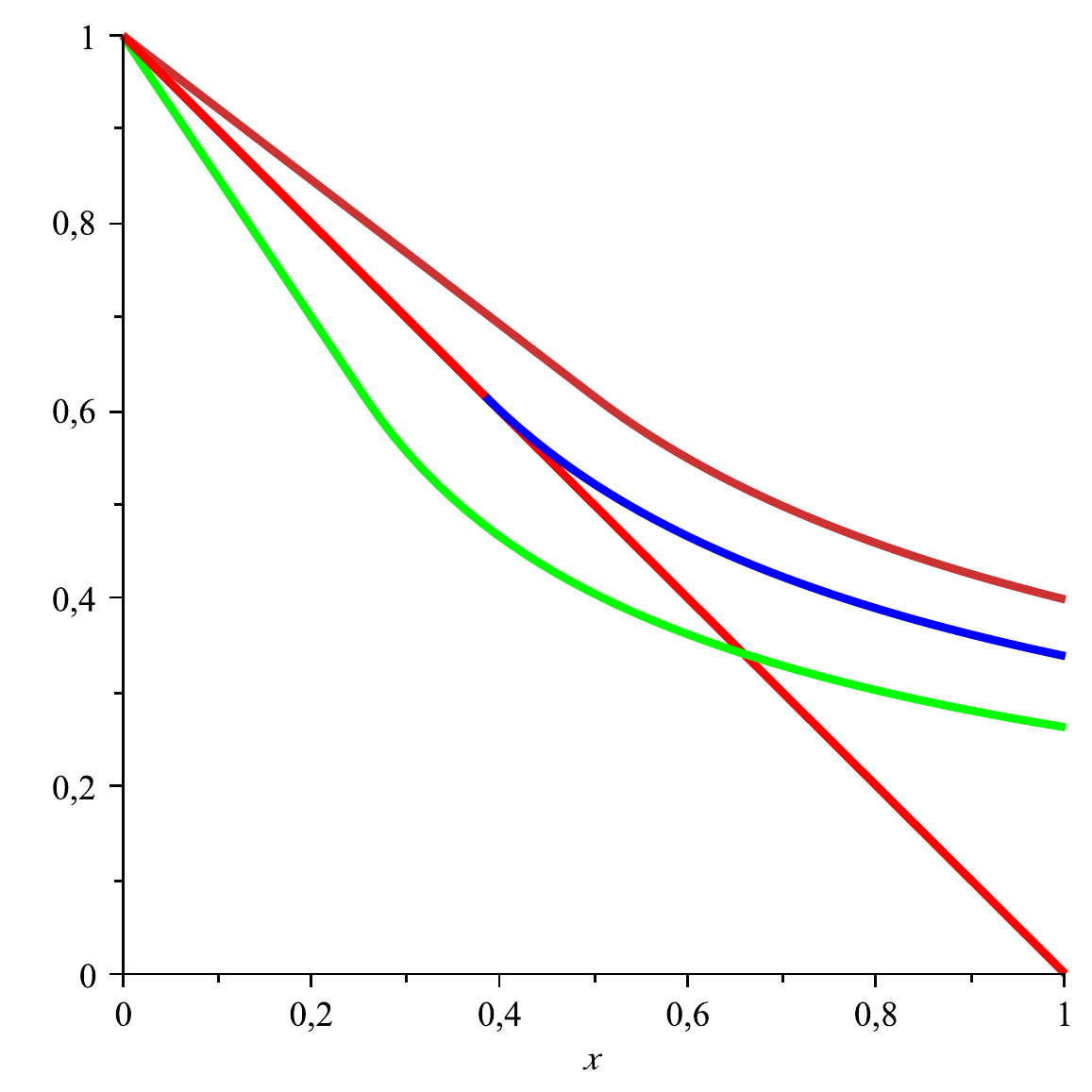}
\caption{Candidate value functions $V_{\ov S}$ for $\ov S=(0,x_i],\;i=1,2,3$, where $x_1=0.25<x_2=0.38=x^*<x_3=0.5$.}
\label{fig1}
\end{figure}

Figure \ref{fig1} shows the graphs of three candidate value functions
of the form \eqref{eq:cand_riesz} for the sets $\ov S_i=(0,x_i]$,
  where $x_1<x_2=x^*=\gamma K/(1+\gamma)<x_3$. As explained in Section
  2.3, in contrast to candidate functions of the form
  \eqref{eq:cand_orig}, the functions are differentiable. It turns out
  that the true value function is the only candidate function that
  coincides with the reward function at the boundary point $x_i$, in
  accordance  with Theorem \ref{thm:uniqueness}.

\subsection{Explicit form of the integral equation in case $d=2$}

By parametrizing the boundary of the stopping set by a curve $x_2=\gamma(x_1)$, we can rewrite equation \eqref{eq:int_eqn} for the case $d=2$ more explictly as
\begin{equation}
\label{formulad2}
K-x_1-\gamma(x_1)=\int_0^{x_1^*}\int_0^{\gamma(y_1)}G_r(x_1,\gamma(x_1),y_1,y_2)\hat\sigma(y_1,y_2)m(dy_1,dy_2),
\end{equation}
where $x_1^*$ is the optimal value in the one-dimensional case given in \eqref{eq:opt_value_1d}. 
To make this integral equation explicit an expression for the Green kernel is needed.
To this end, let $X=(X^{(1)}_t,X^{(2)}_t)_{t\geq 0}$ be a 2-dimensional geometric
Brownian motion started from $(x_1,x_2)\in\R_+^2$ as introduced in
Section 2.1.
The joint density of $(W^{(1)}_t,W^{(2)}_t)$ with
respect to the Lebesgue measure is given by 
\begin{equation}
\label{N1}
f_W(x,y;t)=\frac 1{2\pi t\sqrt{1-\rho^2}}\exp\left(-\frac 1{2t(1-\rho^2)}
B_\rho(x,y)\right),
\end{equation}
where $\rho:=\E\left(W^{(1)}_1\,W^{(2)}_1\right)$ and 
$$
B_\rho(x,y):=x^2-2\rho xy+y^2\geq 0.
$$
Introduce
$$
Z^{(1)}_t:= W^{(1)}_t + m_1\,t,\qquad Z^{(2)}_t:= W^{(2)}_t + m_2\,t
$$
with
$$
m_1:= (\mu_1-\frac 12a^2_1)/a_1,\qquad 
m_2:= (\mu_2-\frac 12a^2_2)/a_2.
$$
The joint density of $(Z^{(1)}_t,Z^{(2)}_t)$ is obtained from
(\ref{N1}):
\begin{align*}
f_Z(x,y;t)&=\frac 1{2\pi t\sqrt{1-\rho^2}}\exp\left(-\frac 1{2t(1-\rho^2)}
B_\rho(x-m_1t,y-m_2t)\right)\\
&=\exp\left(-\frac 1{2(1-\rho^2)}
\left(A_\rho(x,y;m_1,m_2)+t\, B_\rho(m_1,m_2)\right)\right)
\,f_W(x,y;t)
\end{align*}
with
$$
A_\rho(x,y;m_1,m_2):=2\rho(m_2x+m_1y)-2(m_1x+m_2y).
$$
Formula (29) in Erdelyi et al. p. 146 \cite{erdelyi54} yields
\begin{equation}
\label{R1}
\int_0^\infty {\rm e}^{-r t} f_W(x,y;t)dt
=
\frac 1{\pi \sqrt{1-\rho^2}}
\,K_0\left(\sqrt{r}\sqrt{\frac{2B_\rho(x,y)}
{1-\rho^2}}\right),
\end{equation}
where $K_0$ is the modified Bessel function of second kind given by
(see formula 9.6.21 in Abramowitz and Stegun p. 376
\cite{abramowitzstegun70} and for other
formulas, e.g., 9.6.13) 
$$
K_0(u)=\int_0^\infty\frac{\cos(uv)}{\sqrt{1+v^2}}\,dv.
$$
%

\bigskip 
\noindent
To find the resolvent kernel for $X$
consider for positive $u$ and $v$ 
\begin{align*}
\Pro_{(x_1,x_2)}\left(X^{(1)}_t\leq u,\, X^{(2)}_t\leq v\right)
&=
\Pro\left(Z^{(1)}_t\leq \hat u,\, Z^{(2)}_t\leq \hat v\right)
\end{align*}
with
$$
\hat u:= \frac 1{a_1}\log{\frac u{x_1}},\qquad
\hat v:= \frac 1{a_2}\log{\frac v{x_2}}.
$$
Consequently,
\begin{align*}
f_X(u,v;x_1,x_2;t)&:=
\Pro_{(x_1,x_2)}\left(X^{(1)}_t\in du,\, X^{(2)}_t\in dv\right)/\,du\,dv
\\
&=
\frac 1{a_1a_2 uv}\,f_Z(\hat u,\hat v;t),
\end{align*}
and, hence, from (\ref{R1}) we obtain the following expression for the resolvent kernel of $X$ with respect to the Lebesgue measure
\begin{align*}
&G_r(u,v;x_1,x_2)
\\
&\hskip.1cm :=
\int_0^\infty {\rm e}^{-r t} f_X(u,v;x_1,x_2;t)dt
\\
&\hskip.1cm=
\frac 1{\pi \sqrt{1-\rho^2}}
\frac 1{a_1a_2 uv}
\exp\left(-\frac 1{2(1-\rho^2)}
A_\rho(\hat u,\hat v;m_1,m_2)\right)
\,K_0\left(\sqrt{\hat r}\sqrt{\frac{2B_\rho(\hat u,\hat v)}
{1-\rho^2}}\right),\\
\end{align*}
where
$$
\hat r:=r + \frac{B_\rho(m_1,m_2)}
{2(1-\rho^2)}.
$$
Plugging this expression of the Green kernel in \eqref{formulad2} (where we now have taken $m$ to be the Lebesgue measure on $\R^2$) yields an explicit integral equation having the stopping boundary as a unique solution. 
 

{\subsection{On the numerical solution to the integral equation}
The numerical solution to the integral equation \eqref{eq:int_eqn} does not seem to be standard. Here, we present an ad hoc method that works fine for the optimal investment problem. We concentrate on the case $d=2$ although a similar approach could be applied in higher dimensions also. The basic observation is that the equation 
\begin{equation}\label{eq:2-dim}
K-x-\gamma(x)=\int_0^{x_1^*}\int_0^{\gamma(y_1)}G_r(x,\gamma(x),y_1,y_2)\sigma(y_1,y_2)m(dy),~~x\in [0,x_1^*]
\end{equation}
for the unknown boundary $\gamma$ can be written in fixed point form as
\begin{equation*}
\gamma(x)=T(\gamma)(x),~~x\in [0,x_1^*],
\end{equation*}
where the (nonlinear) operator $T$ is given by
\[T(\gamma)(x)=K-x-\int_0^{x_1^*}\int_0^{\gamma(y_1)}G_r(x,\gamma(x),y_1,y_2)\sigma(y_1,y_2)m(dy).\]
One can start with a first approximation $\gamma_1$ for the unknown boundary $\gamma$ (for example, the straight line described in \cite{ho}). Using a suitable discretization of the state space and the integral, the iterated function sequence
\[\gamma,~T(\gamma),~T(T(\gamma)),...\]
turns out to converge fast to an approximate solution to \eqref{eq:2-dim}.

\begin{figure}[t]
\center
\includegraphics[width=8cm]{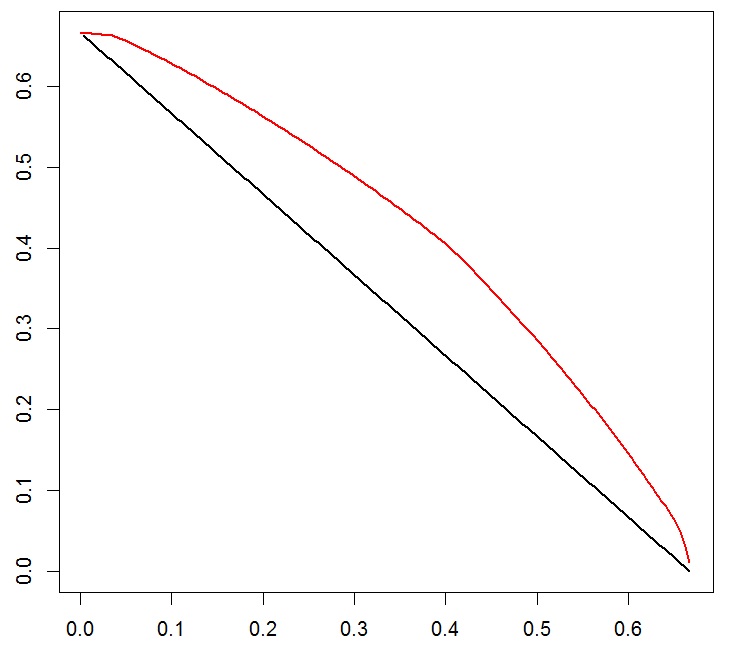}
\caption{Boundary of the optimal stopping set $S$ in the case $d=2$ for parameters $\mu_1=\mu_2=r=a^2_1=a_2^2=K=1$, $\rho=0$ (red) and the straight line discussed in \cite{ho} (black).}
\label{fig2}
\end{figure}
}

\subsection{Solution for spectrally negative geometric L\'evy processes}
By a careful inspection of the previous proofs, it turns out that the
results of the previous subsections can be generalized to an
underlying jump process. To be more precise, we consider the optimal
stopping problem \eqref{eq:value} for an underlying geometric L\'evy
process {
\[X=(({\rm e}^{Y^{(1)}_t},...,{\rm
  e}^{Y^{(d)}_t}))_{t\geq 0},\]
where $(Y_t^{(i)})_{t\geq 0},\, i=1,\dots, d,$ are spectrally negative L\'evy
processes, that is their L\'evy measures are concentrated on $\R_-$. Note that
since the co-ordinates of $X$ may have negative jumps,} overshoot may occur for the optimal stopping time in \eqref{eq:value}.

For our approach to work, we need to assume that $0$ is regular for $(-\infty,0)^d$ for the underlying $d$-dimensional L\'evy process $Y$, i.e.
\begin{equation}\label{eq:ass_regular}
\inf\{t>0:Y^{(i)}_t<0\mbox{ for all }i\}=0~\Pro_0\mbox{-a.s.} .
\end{equation}
This assumption guarantees that the value
function is $C^1$ across the boundary, which is highly related to the
regularity of the stopping set $S$, see below and the discussion in \cite{AK,CI}
for the one-dimensional case. Then, we obtain the following
generalization of Theorem \ref{fund_prop} 
and Theorem \ref{thm:uniqueness}:

\begin{thm} 
\label{thm:snlevy}
Under the assumptions made above, it holds that for all ${\bf x}\in \R_+^d$
\begin{equation*}
V({\bf x})=\int_SG_r({\bf x},{\bf y})\hat\sigma({\bf y})m({\bf dy}),
\end{equation*}
where \[\hat\sigma({\bf y}):=(r-{\cal G})g({\bf y}).\]
Furthermore, for any nonempty, south-west connected, convex set $\ov S\subseteq \{{\bf x}\in\R_+^d:\sum_{i=1}^dx_i<K\}$ such that 
\begin{equation*}
g({\bf x})=\int_{\ov S}G_r({\bf x},{\bf y})\hat\sigma({\bf y})m({\bf dy})\mbox{ for all ${\bf x}\in \partial \ov S$}
\end{equation*}
it holds that $\ov S=S$. 
\end{thm}

\begin{proof}[Sketch of a proof]
As indicated above, the proof is analogous to the proof in the case without jumps. Therefore, we only mention the few changes in the arguments:
\begin{enumerate}
\item Assumption \eqref{eq:ass_regular} is needed to make the proof of Lemma \ref{lem:smooth_fit} work. Indeed, it immediately implies that for all $\zeta\in\R^d$
\[\inf\{t>0:Y^{(i)}_t<-\log(1+\epsilon\zeta_i)\mbox{ for all }i\}\rightarrow 0~\Pro_0\mbox{-a.s. as }\epsilon\searrow 0\]
which is easily seen to assure that
\[
\tau_\epsilon\rightarrow 0\;\;\mbox{a.s.}
\]
in the notation of the proof of Lemma \ref{lem:smooth_fit}. This assures smooth fit to hold as for the Brownian motion case.

\item For the proof of Proposition \ref{prop:repr_measure_infinite},
  one has to use the more general approximation result given in
  \cite[Theorem 2.1]{OkSu} instead of \cite[Appendix
    D]{Oks}. Furthermore, we note that due to the assumptions on the
  jumps, it holds on the interior of $S$ that 
\[(r-{\cal G})V=(r-{\cal G})g
\]
(although ${\cal G}$ is a non-local operator).
\end{enumerate}
\end{proof}

\section{American put option}\label{sec:put}
 
 { As pointed out in the introduction, the problem of
   pricing the American put
option can be seen as the investment problem  under 
a finite time horizon with $d=1$. Therefore, we found it motivated to briefly
discuss our approach  also in this problem setting. To be more
specific, we demonstrate how  the Riesz representation can
   be used to derive  the integral equation which characterizes the early exercise
 boundary and, furthermore, to prove the
 uniqueness of the solution
of the equation. 

An important vehicle in the proof of the uniqueness of the solution of
the integral equation in
\cite{Peskir2005_2} (see also \cite{PS}) is  an extension of the
It\^o-Tanaka formula developed in  \cite{Peskir05} 
  and called the local time-space formula. This formula is needed to be able
  to analysis non-smooth candidate solutions. Our proof of the uniqueness
  uses similar ideas as presented in  \cite{Peskir2005_2} but the
  advantage with the Riesz representation is that we can operate with
  {\it smooth} candidates (cf. Section 2.3 above).} 

Let $X=(X_t)_{t\geq 0}$ denote the stock price process in the Black-Scholes model and let
$\widehat\Pro$ denote the martingale measure. Hence, under
$\widehat\Pro,$  $X$ is a geometric Brownian motion with volatility
$a^2>0$ and drift $r$ which is also the risk-free interest rate on
the market. The fair price of the American
put option  is given by 
\begin{equation}
\label{ACCprice}
p^*:=\sup_{\tau\in{\cal M}_{0,T}}\widehat\E_x\left({\rm e}^{-r\tau}(K-X_\tau)^+\right),
\end{equation}
where $T$ is the expiration time, ${\cal M}_{s,T}$ is for $0\leq s< T$ the set of stopping times 
with values in $(s,T],$ $x$ is the initial price of the stock, and $K$ 
 is the strike price.

Consider now the value of the OSP associated with the American put
given via
\begin{equation}
\label{ACCprice2}
V(s,x):=\sup_{\tau\in{\cal M}_{s,T}}\widehat\E_{(s,x)}\left(\e^{-r(\tau-s)}(K-X_\tau)^+\right).
\end{equation}
From the general theory of optimal stopping we know that $V$ is
$r$-excessive { for the space-time process $\bar X=((t,X_t))_{0\leq
  t\leq T}$ } and that the stopping region consists of the points,
where the value equals the reward, i.e.,
$$
S:=\left\{ (s,x): V(s,x)=(K-x)^+\right\}.
$$

The results in the next theorem are well-known. Our interest hereby is
focused only on (iv) and (v) but for the readability we also
indicate references for    (i), (ii), and (iii).


\begin{thm}
\label{ACC1}
There exists a 
function $s\mapsto b(s), 0\leq s\leq T,$ such that
\begin{enumerate}[{\rm (i)}]
\item\ $b(0)<K,\quad b(T)=K,$\label{item:1}
 {\item\ $b$ is increasing, convex and differentiable in $(0,T),$\label{item:2}}
\item\ $ S=\{(s,x)\,:\, x\leq b(s)\}$ and
$\tau^*:=\inf\{t:X_t\leq b(t)\}$ is an optimal stopping time,\label{item:3}
\item\ the price of the option at time $s$ when $X_s=x$ has the 
unique  { Doob-Meyer} decomposition\label{item:4}
$$
V(s,x)=rK\int_s^T\e^{-r(t-s)}\widehat \Pro_{s,x}(X_t<b(t))dt + 
\widehat \E_{s,x}(\e^{-r(T-s)}(K-X_T)^+),
$$
 { where the first term on the right hand side (the finite
   variation part of the decomposition) is called
   the early exercise premium,} 
\item\ for any $s>0$ the function $b(t),\ t\geq s,$ is the \label{item:5}
unique continuous solution of the integral 
equation 
\begin{align}
\label{ACCeq}
K-b(s)=&\ rK\int_s^T\e^{-r(t-s)}\widehat \Pro_{(s,b(s))}(X_t<b(t))dt\\ &\qquad+ 
\widehat \E_{(s,b(s))}(\e^{-r(T-s)}(K-X_{T})^+).\nonumber
\end{align}
\end{enumerate}
\end{thm}
\begin{proof}
For the existence of an increasing and continuous curve $b$ such that
\eqref{item:1} and \eqref{item:3} hold, see Jacka \cite{jacka91} (and also Myneni \cite{Myneni},
Karatzas \cite{Karatzas97}, 
and Karatzas and Shreve \cite{KaShreve98}). { For the differentiablity,
see Chen and Chadam \cite{CC06} and for the convexity, see
Chen, Chadam, Jiang and Zheng \cite{CCJZ} and  Ekstr\"om \cite{ekstrom04}.}

%
%
For \eqref{item:4} we recall from the Riesz representation that there
exists a unique $\sigma$-finite measure on $(0,T]\times\R_+$ such that 
\begin{equation}
\label{value3}
V(s,x)= \int\int_{(0,T]\times\R_+}\bar
G_r\left((s,x),(t,y)\right)\sigma(dt,dy).
\end{equation}
{It is also well-known (see \cite{jacka91}, \cite{PS}) 
or
follows as in Lemma \ref{lem:smooth_fit} that $\partial V/\partial x$
exists and is continuous on $(0,T]\times \R_+$ (smooth fit holds). The differentiability of $b$ implies that
also  $\partial V/\partial t$ exists and is continuous (see
\cite{Myneni} p. 16 where a reference to Moerbeke \cite{moerbeke76}
Lemma 5 is given).} For the absolute
continuity of $v:=\partial V/\partial x$ as a function of $x$, note
that $v$ is $C^1$ everywhere (but not on $b(t)$) with bounded derivative, and therefore Lipschitz continuous as a function of $x$.
Hence, using  Proposition \ref{prop:space-time-measure} we have for $(t,y)\in \inte(S)$
$$
\sigma(dt,dy)=rK\, dt\, m(dy).\ 
$$ 
Since $V(T-,x)=(K-x)^+$, we furthermore have
$$
\sigma({T},dy)=(K-y)\, m(dy)\ {\rm for}\ 0<y<K,
$$ 
and otherwise $\sigma=0.$ Hence using \eqref{eq:green_time_space}
\begin{align*}
\label{value3}
V(s,x)&=\int\int_{(0,T]\times\R_+}\bar
G_r\left((s,x),(t,y)\right)\sigma(dt,dy)\\
&=\int\int_{ \inte(S)}\bar
G_r\left((s,x),(t,y)\right)rK\, dt\, m(dy)\\
&\qquad\qquad +\int\int_{\{T\}\times\R_+}\bar
G_r\left((s,x),(t,y)\right)(K-y)^+\, \delta_{\{T\}}(dt)m(dy)\\
&=\int\int_{ \inte(S)}\bar
G_r\left((s,x),(t,y)\right)rK\, dt\, m(dy)\\
&\qquad\qquad +\int_{\R_+}\bar
G_r\left((s,x),(T,y)\right)(K-y)^+\, m(dy)\\
&=rK\int_s^T\e^{-r(t-s)}\widehat \Pro_{s,x}(X_t<b(t))dt + 
\widehat \E_{s,x}(\e^{-r(T-s)}(K-X_T)^+).
\end{align*}
This proves \eqref{item:4}.

\eqref{item:5} can now be proved by following the proof of Theorem \ref{thm:uniqueness}. More precisely, one considers a continuous candidate 
boundary function $\ov b$ that also fulfills 
\begin{equation*}
K-\ov b(s)=rK\int_s^T\e^{-r(t-s)}\widehat \Pro_{(s,\ov b(s))}(X_t<\ov b(t))dt + 
\widehat \E_{(s,\ov b(s))}(\e^{-r(T-s)}(K-X_{T})^+)
\end{equation*}
and defines the associated candidate value function 
\begin{equation}
\ov V(s,x):= \int\int_{\{(t,y):y\leq \ov b(t)\}}\bar
G_r\left((s,x),(t,y)\right)\ov \sigma(dt,dy),
\end{equation}
where $\ov\sigma$  for $(t,y)\in (0,T)\times \R_+$ such that
$y\leq \ov b(t)$ is given by
$$
\ov \sigma(dt,dy)=rK\, dt\, m(dy) 
$$ 
and for $(t,y)\in \{T\}\times (0,K)$
$$
\ov \sigma({T},dy)=(K-y)\, m(dy).
$$ 
Then one proves -- using the same arguments as in the proof of Theorem
\ref{thm:uniqueness} -- the following four steps:
\begin{enumerate}
\item $\ov V(s,x)=(K-x)^+$ for all $(s,x)\in (0,T)\times \R_+ \mbox{ with }x\leq \ov b(s)$,
\item $\ov V\leq V$,
\item $\ov b \geq b$,
\item $\ov b\leq b$.
\end{enumerate}
{ In particular, a space-time version of formula (\ref{dynkin_1}) in Lemma \ref{lem:integr_stop} is
needed. To formulate this, 
let $s: (0,T)\times\R_+\rightarrow[0,\infty)$ be a measurable function and
define
\[w(u,x):=\int_u^T\int_{\R_+} G_r((u,x),(t,y))\,s(t,y)\, dt\, m(dy).\]
where $u\in(0,T)$ is fixed. Then for each stopping time $\tau$  taking values in $(u,T)$  and each ${ x}\in\R_+$
\[w(u,x)=\E_{(u,x)}\left({\rm e}^{-r(\tau-u)}w(\tau,X_\tau)\right)+\E_{(u,x)}\left(\int_\tau^T {\rm e}^{-r(t-\tau)}s(t,X_t)dt\right).\]
We leave the details of the proofs of this formula and steps 1-4 to the reader.
}
\end{proof}


\appendix
\section{Appendix}
\begin{proof}[{\bf A1} Proof of Proposition \ref{prop:self-dual}] Let $X$ be a
  $d$-dimensional geometric Brownian motion with parameters
  ${\boldsymbol \mu}$ and {${\boldsymbol R}$} started at ${\bf 1}$ { as introduced in Section \ref{subsec:riesz_infinite}, and denote the Lebesgue density of $X_t$ by $f_{X_t}$.
  }
  
 In the following, we understand each operation (as multiplication, division etc.) componentwise. Since $1/X$ is a geometric Brownian motion 
 {a short calculation yields}
\[f_{X_t}(1/{\bf y})=f_{1/X_t}({\bf y})\exp\left(2\sum_{i=1}^d\log(y_i)\right)\]
and
\[f_{1/X_t}({\bf y})=f_{X_t}({\bf y})\exp\left(-2\,{\boldsymbol{ \ov{\mu}}}\,{\boldsymbol \Sigma}^{-1}\log ({\bf y})^{tr}\right)
\]
{where $\boldsymbol{ \ov{\mu}}=\big((\mu_i-\frac{1}{2}a_i^2)/a_i\big)_{i=1}^d$.
}
Let $m$ be the measure on $\R_+^d$ with Lebesgue density
\[h({\bf y})=\exp\left(-\sum_{i=1}^d\log(y_i)+2\,{\boldsymbol{\ov{\mu}}}\,{\boldsymbol \Sigma}^{-1}\log( {\bf y})^{tr}\right),\]
then for all nonnegative and measurable functions $f$ and $g$, it holds that
\begin{align*}
\hskip-1cm&\int_{\R_+^d} f({\bf x})G_rg({\bf x})m({\bf dx})\\
=&\int_0^\infty {\rm e}^{-rt}\int_{\R_+^d}\int_{\R_+^d} f({\bf x})g({\bf x}{\bf z})h({\bf x})f_{X_t}({\bf z}){\bf dx} {\bf dz}dt\\
=&\int_0^\infty {\rm e}^{-rt}\int_{\R_+^d}\int_{\R_+^d} f({\bf w}/{\bf z})g({\bf w})h({\bf w}/{\bf z})f_{X_t}({\bf z})\exp\left(-\sum_{i=1}^d\log(z_i)\right){\bf dw} {\bf dz}dt\\
=&\int_0^\infty {\rm e}^{-rt}\int_{\R_+^d}\int_{\R_+^d} f({\bf y}{\bf w})g({\bf w})h({\bf y}{\bf w})f_{X_t}(1/{\bf y})\exp\left(-\sum_{i=1}^d\log(y_i)\right){\bf dy} {\bf dw}dt\\
=&\int_0^\infty {\rm e}^{-rt}\int_{\R_+^d}\int_{\R_+^d} f({\bf z}{\bf y})g({\bf z})h({\bf z}{\bf y})f_{1/X_t}({\bf y})\exp\left(\sum_{i=1}^d\log(y_i)\right){\bf dz} {\bf dy}dt\\
=&\int_0^\infty {\rm e}^{-rt}\int_{\R_+^d}\int_{\R_+^d} f({\bf z}{\bf y})g({\bf z})h({\bf z})h({\bf y})f_{X_t}({\bf y})\times\\
&\hspace{2cm}\times\exp\big(-2\,{\boldsymbol{\ov{\mu}}}\,{\boldsymbol \Sigma}^{-1}\log ({\bf y})^{tr}\big)\exp\left(\sum_{i=1}^d\log(y_i)\right){\bf dz} {\bf dy}dt\\
=&\int_0^\infty {\rm e}^{-rt}\int_{\R_+^d}\int_{\R_+^d} f({\bf z}{\bf y})g({\bf z})h({\bf z})f_{X_t}({\bf y}){\bf dz} {\bf dy}dt\\
=&\int_{\R_+^d} G_rf({\bf z})g({\bf z})m({\bf dz}).
\end{align*}

\end{proof}

\begin{proof}[{\bf A2} Proof of Proposition \ref{prop:repr_measure_infinite}]
By noting that convex sets have a Lipschitz surfaces there exists (see, e.g., \cite[Appendix D]{Oks}) a sequence $(u_j)_{j\in\N}$ of $C^2$-functions such that $u_j\rightarrow u$ uniformly on compacts and ${\cal G}u_j\rightarrow {\cal G}u$ uniformly on compact subsets of $\R_+^d\setminus \partial D$. 
{
More precisely, this sequence is given by convolution with a sequence of $C^\infty$-functions $\eta_j,\;j\in\N,$ with compact support:
\[u_j({\bf x})=u*\eta_j({\bf x})=\int_{\R^d}u({\bf x}-{\bf y})\eta_j({\bf y}){\bf dy},\]
where $*$ denotes convolution. Furthermore, by the proof of
\cite[Appendix D]{Oks} it is immediate that
$({\cal G}-r)u_j=(({\cal G}-r)u)*\eta_j.$ Therefore, the boundedness assumption on
$({\cal G}-r)u$ implies that the sequence $({\cal G}-r)u_j,\;j\in\N$, is uniformly locally bounded around $\partial S$.  \\
}
Now, take a sequence $(K_n)_{n\in \N}$ of compact subsets of $\R_+^d$ such that $\tau_{n}:=\inf\{t\geq 0: X_t\not\in K_n\}\rightarrow \infty$ as $n\rightarrow\infty$. By Dynkin's formula 
\[u_j({\bf x})=\E_{\bf x}\left({\rm e}^{-r\tau_n}u_j(X_{\tau_n})\right)-\E_{\bf x}\left(\int_0^{\tau_n}{\rm e}^{-rs}({\cal G}-r)u_j(X_s)ds\right).\]
Letting $j\rightarrow\infty$ and keeping in mind that the Green measure is absolutely continuous with respect to the Lebesgue measure, we obtain 
\[u({\bf x})=\E_{\bf x}\left({\rm e}^{-r\tau_n}u(X_{\tau_n})\right)-\E_{\bf x}\left(\int_0^{\tau_n}{\rm e}^{-rs}({\cal G}-r)u(X_s)ds\right)\]
{by dominated convergence
and the convention that $({\cal G}-r)u({\bf x})=0$ (or any other value) on the null set $\partial S$. 
}
Letting $n\rightarrow\infty$ and using the boundedness of $u$, we have 
\begin{align*}
\int G_r({\bf x},{\bf y})\sigma({\bf dy})&=u({\bf x})=\E_{\bf x}\left(\int_0^{\infty}{\rm e}^{-rs}(-{\cal G}+r)u(X_s)ds\right)\\&=\int G_r({\bf x},{\bf y})(r-{\cal G})u({\bf y})m({\bf dy})
\end{align*}
and, by the uniqueness of the integral representation, 
\[\sigma({\bf dy})=(r-{\cal G})u({\bf y})m({\bf dy}).\]
\end{proof}

\begin{proof}[{\bf A3} Proof of the duality; Proposition \ref{prop:space-time-riesz}]
The claim is that there exists a resolvent kernel ${\widehat G}$ such that (\ref{space-time-dual}) holds. Consider
for non-negative and measurable $f$ and $g$ 
\begin{align*}
&
\int_0^T ds \int_{\R_+}m(dx) f(s,x){\bar G}_rg(s,x)
\\
&\hskip1.5cm=
\int_0^T ds \int_{\R_+}m(dx) f(s,x)\left[\int_s^T dt \int_{\R_+}m(dy) {\bar G}_r((s,x),(t,y))g(t,y)\right]
\\
&\hskip1.5cm=
\int_0^T dt \int_{\R_+}m(dy) g(t,y)\left[\int_0^t ds \int_{\R_+}m(dx) {\bar G}_r((s,x),(t,y))f(s,x)\right].
\end{align*} 
Let
\begin{align*}
{\widehat G}_rf(t,y)&:=\int_0^t ds \int_{\R_+}m(dx) {\bar G}_r((s,x),(t,y))f(s,x).
\\
&
=\int_0^t ds \int_{\R_+}m(dx)\,{\rm e}^{-r\, (t-s)}p(t-s;x,y)f(s,x).
\end{align*} 
We claim that ${\widehat G}_r$ constitutes a resolvent kernel as
defined in \cite[p. 493]{KW}. Conditions (a), (b) and (d) therein are easily
verified. It remains to check condition (c), i.e., the resolvent equation
$$
{\widehat G}_\alpha-{\widehat G}_\beta=(\beta-\alpha){\widehat G}_\alpha{\widehat G}_\beta.
$$
Indeed, 
\begin{align*}
{\widehat G}_\alpha({\widehat G}_\beta f)(t,y)
&=\int_0^t ds \int_{\R_+}m(dx)\,\e^{-\alpha\,(t-s)}p(t-s;y,x){\widehat G}_\beta f(s,x)
\\
&=\int_0^t ds \int_{\R_+}m(dx)\,\e^{-\alpha\, (t-s)}p(t-s;y,x)
\\
&\hskip2cm
\times \int_0^s du \int_{\R_+}m(dz)\,\e^{-\beta\, (s-u)}p(s-u;x,z) f(u,z)
\\
&=\int_0^t ds\int_0^s du\, \e^{-\alpha\,
 (t-s) -\beta(s-u)} \int_{\R_+}m(dz) p(t-u;y,z) f(u,z)
\\
&=\int_0^t du\int_u^t ds\, \e^{(\alpha-\beta)s}\,\e^{-\alpha t+\beta u}
 \int_{\R_+}m(dz) p(t-u;y,z) f(u,z)
\\
&={\frac {1}{\alpha-\beta}}\left({\widehat G}_\beta f(t,y)-{\widehat G}_\alpha f(t,y)\right),
\end{align*} 
as claimed.  
\end{proof}


\begin{proof}[{\bf A4} Proof of Lemma \ref{lem:smooth_fit}]
Let ${\bf x_0}\in\partial S$ and ${\boldsymbol \zeta}\in\R^d$ with { Euclidian norm }1. Then $V({\bf x_0})=g({\bf x_0})$ and $V\geq g$. Therefore, for $\epsilon$ such that ${\bf x_0}+\epsilon{\boldsymbol \zeta}\in\R_+^d$ it holds 
\[V({\bf x_0}+\epsilon{\boldsymbol \zeta})-V({\bf x_0})\geq g({\bf x_0}+\epsilon{\boldsymbol \zeta})-g({\bf x_0}),\]
which yields 
\begin{align*}
\liminf_{\epsilon\rightarrow0}\frac{1}{\epsilon}(V({\bf x_0}+\epsilon{\boldsymbol \zeta})-V({\bf x_0}))&\geq \liminf_{\epsilon\rightarrow0}\frac{1}{\epsilon}(g({\bf x_0}+\epsilon{\boldsymbol \zeta})-g({\bf x_0}))\\
&=D_{\boldsymbol \zeta} g({\bf x_0})=-\sum_{i=1}^d\zeta_i.
\end{align*}
On the other hand, let $\tau^*=\inf\{t\geq 0:X_t\in S\}$ denote the optimal stopping time and write
\[\tau_\epsilon=\inf\{t\geq 0:({\bf x_0}+\epsilon{\boldsymbol \zeta})X_t\in S\}.\]
Then (using componentwise multiplication of vectors and the notation ${\bf 1}=(1,...,1)$)
\[V({\bf x_0}+\epsilon{\boldsymbol \zeta})=\E_{{\bf x_0}+\epsilon{\boldsymbol \zeta}}\left({\rm e}^{-r\tau^*}g(X_{\tau^*})\right)=\E_{{\bf 1}}\left({\rm e}^{-r\tau_\epsilon}g(({\bf x_0}+\epsilon{\boldsymbol \zeta})X_{\tau_\epsilon})\right)\]
and -- since ${\bf x_0}\in S$ --
\[V({\bf x_0})\geq \E_{{\bf 1}}\left({\rm e}^{-r\tau_\epsilon}g({\bf x_0}X_{\tau_\epsilon})\right).\]
{
By  Lemma \ref{lem:properties_S}, the set
\[F:=\{{\bf x}\in(0,\infty)^d:{\bf x}\leq {\bf x_0}\},\]
is a subset of $S$, hence - keeping the continuity of the sample paths in mind~-
\[\tau_\epsilon\leq \inf\{t\geq 0:({\bf x_0}+\epsilon{\boldsymbol \zeta})X_t\in F\}\rightarrow 0\mbox{ under }\Pro_{{\bf 1}}\mbox{ as }\epsilon\rightarrow0.\]
}
Consequently,
\begin{align*}
\frac{1}{\epsilon}(V({\bf x_0}+\epsilon{\boldsymbol \zeta})-V({\bf x_0}))&\leq \frac{1}{\epsilon} \E_{{\bf 1}}\left({\rm e}^{-r\tau_\epsilon}(g(({\bf x_0}+\epsilon{\boldsymbol \zeta})X_{\tau_\epsilon})-g({\bf x_0}X_{\tau_\epsilon}))\right)\\
&=\frac{1}{\epsilon} \E_{{\bf 1}}\left({\rm e}^{-r\tau_\epsilon}(-\epsilon\sum_i\zeta_iX^{(i)}_{\tau_\epsilon})\right)\\
&\rightarrow -\sum_i\zeta_i=D_{\boldsymbol \zeta} g({\bf x_0})\mbox{ as }\epsilon\rightarrow0.
\end{align*}
This proves 
{
that the value function $V$ is differentiable on $\partial S$. Furthermore, $V\in C^1$ on $(0,\infty)^d\setminus \partial S$, and by the proof of Lemma \ref{lem:properties_S}, $V$ is convex. Since differentiable convex functions are $C^1$ (see \cite[Theorem 2.2.2]{Borwein2010}), we obtain
the first claim. For the second one, note that $({\cal G}-r)V=0$ on $S^c$ and $({\cal G}-r)V({\bf y})=({\cal G}-r)g({\bf y})=-\left(rK+\sum_{i=1}^d(\mu_i-r)y_i\right)$ on $int(S)$, which gives the locally boundedness around $\partial S.$
}
\end{proof}
%
%


\begin{proof}[{\bf A5} Proof of Lemma \ref{lem:integr_stop}]
By the definition of the Green kernel and the strong Markov property it holds that for all ${\bf x}\in\R_+^d$ and all stopping times~$\tau$
\begin{align*}
w({\bf x})&=\int G_r({\bf x},{\bf y})s({\bf y})m({\bf dy})=\E_{\bf x}\left(\int_0^\infty {\rm e}^{-rt}s(X_t)dt\right)\\
&=\E_{\bf x}\left(\int_\tau^\infty {\rm e}^{-rt}s(X_t)dt\right)+\E_{\bf x}\left(\int_0^\tau {\rm e}^{-rt}s(X_t)dt\right)\\
&=\E_{\bf x}\left({\rm e}^{-r\tau}\E_{\bf x}\left(\int_\tau^\infty {\rm e}^{-r(t-\tau)}s(X_t)dt|\mathcal{F}_\tau\right)\right)+\E_{\bf x}\left(\int_0^\tau {\rm e}^{-rt}s(X_t)dt\right)\\
&=\E_{\bf x}\left({\rm e}^{-r\tau}\E_{X_\tau}\left(\int_0^\infty {\rm e}^{-rt}s(X_t)dt\right)\right)+\E_{\bf x}\left(\int_0^\tau {\rm e}^{-rt}s(X_t)dt\right)\\
&=\E_{\bf x}\left({\rm e}^{-r\tau}w(X_\tau)\right)+\E_{\bf x}\left(\int_0^\tau {\rm e}^{-rt}s(X_t)dt\right).
\end{align*}
\end{proof}
\noindent
\section*{Acknowledgement} 
\noindent 
Paavo Salminen thanks the 
    Mathematisches Seminar at  Christian-Albrechts-Universit\"at
for the hospitality and the support during the stay in Kiel.

\bibliographystyle{imsart-number}
\bibliography{3_version_october_15_arxiv}


\end{document}